\documentclass[final]{siamart}

\usepackage{epsfig}
\usepackage{latexsym}
\usepackage{amsmath}
\usepackage{amsfonts}
\usepackage{amssymb}
\usepackage{stmaryrd}
\usepackage{dsfont}
\usepackage{curves}
\usepackage{color}
\usepackage{version}

\newcommand{\tr}{{\rm tr}}
\newcommand{\bas}{{\bf as}}

\newcommand{\bdev}{{\bf dev}}

\renewcommand{\div}{{\rm div}}
\newcommand{\curl}{{\rm curl}}

\newcommand{\R}{\mathds{R}}

\newcommand{\bvarphi}{\mbox{\boldmath$\varphi$}}

\newcommand{\balpha}{\mbox{\boldmath$\alpha$}}

\newcommand{\bepsilon}{\mbox{\boldmath$\varepsilon$}}
\newcommand{\btau}{\mbox{\boldmath$\tau$}}

\newcommand{\bzeta}{\mbox{\boldmath$\zeta$}}
\newcommand{\bgamma}{\mbox{\boldmath$\gamma$}}

\newcommand{\brho}{\mbox{\boldmath$\rho$}}
\newcommand{\bsigma}{\mbox{\boldmath$\sigma$}}

\newcommand{\bSigma}{\mbox{\boldmath$\Sigma$}}
\newcommand{\btheta}{\mbox{\boldmath$\theta$}}

\newcommand{\bnabla}{\mbox{\boldmath$\nabla$}}
\newcommand{\bzero}{{\bf 0}}
\newcommand{\bff}{{\bf f}}
\newcommand{\bv}{{\bf v}}

\newcommand{\bz}{{\bf z}}

\newcommand{\bu}{{\bf u}}

\newcommand{\bg}{{\bf g}}

\newcommand{\bI}{{\bf I}}
\newcommand{\bJ}{{\bf J}}
\newcommand{\bM}{{\bf M}}

\newcommand{\bR}{{\bf R}}

\newcommand{\bV}{{\bf V}}

\newcommand{\bX}{{\bf X}}
\newcommand{\bZ}{{\bf Z}}
\newcommand{\bn}{{\bf n}}

\newcommand{\cA}{{\cal A}}

\newcommand{\cP}{{\cal P}}

\newcommand{\cS}{{\cal S}}
\newcommand{\cT}{{\cal T}}
\newcommand{\cV}{{\cal V}}

\begin{document}

\title{Weakly symmetric stress equilibration and a posteriori error estimation for linear elasticity}

\author{Fleurianne Bertrand%
        \thanks{Institut f\"ur Mathematik, Humboldt-Universit\"at zu Berlin, Unter den Linden 6, 10099 Berlin, Germany
        (\email{fleurianne.bertrand@uni-due.de}).}
        \and Bernhard Kober
        \thanks{Fakult\"at f\"ur Mathematik, Universit\"at Duisburg-Essen, Thea-Leymann-Str. 9, 45127 Essen, Germany
        (\email{bernhard.kober@uni-due.de}, \email{marcel.moldenhauer@uni-due.de}, \email{gerhard.starke@uni-due.de}).
        The authors gratefully acknowledge support by the Deutsche Forschungsgemeinschaft in the Priority Program SPP 1748
        `Reliable simulation techniques in solid mechanics. Development of nonstandard discretization methods, mechanical and
        mathematical analysis' under the project numbers BE 6511/1-1 and STA 402/12-2.}
        \and Marcel Moldenhauer\footnotemark[2]
        \and Gerhard Starke\footnotemark[2]
        }
        
\date{}
\maketitle

\begin{abstract}
  A stress equilibration procedure for linear elasticity is proposed and analyzed in this paper with emphasis on the behavior for
  (nearly) incompressible materials. Based on the displacement-pressure approximation computed with a stable finite element pair,
  it constructs an $H (\div)$-conforming, weakly symmetric stress reconstruction.
  Our focus is on the Taylor-Hood combination of continuous finite element spaces of polynomial degrees $k+1$ and $k$ for
  the displacement and the pressure, respectively.
  Our construction leads to a reconstructed stress tensor by Raviart-Thomas elements of degree $k$ which are weakly symmetric
  in the sense that its anti-symmetric part is zero tested against continuous piecewise polynomial functions of degree $k$.
  The computation is performed locally on a set of vertex patches covering the computational domain in the spirit of equilibration.
  This  weak symmetry allows us to prove that the resulting error estimator constitutes a guaranteed upper bound for the error
  with a constant that depends only on local constants associated with the patches and thus on the shape regularity of the triangulation.
  It does not involve global constants like those from Korn's in equality which may become very large depending on the location and type
  of the boundary conditions. Local efficiency, also uniformly in the incompressible limit, is deduced from the upper bound by the
  residual error estimator. Numerical results for the popular Cook's membrane test problem confirm the theoretical predictions.
\end{abstract}



\headers{Weakly symmetric stress equilibration}
              {F. Bertrand, B. Kober, M. Moldenhauer, and G. Starke}

\section{Introduction}

\label{sec-introduction}

This paper is concerned with a stress equilibration procedure for the displacement-pressure formulation of linear elasticity.
Our emphasis is on the behavior for (nearly) incompressible materials and we concentrate ourselves
on the Taylor-Hood combination of continuous finite element spaces of polynomial degrees $k+1$ and $k$ ($k \geq 1$) for the
displacement and the pressure, respectively. This finite element pair has the advantage that it is conforming for the displacement
approximation which simplifies the derivation of an a posteriori error estimator based on the equilibrated stress. Another property
which will prove to be useful in this context is the fact that the stress, computed directy from the displacement-pressure
approximation, already possesses the convergence order $k$ with respect to the $L^2$-norm.

In contrast to the case of Poisson's equation, where equilibrated fluxes are used, the linear elasticity system involves the
symmetric part of the displacement gradient for the definition of the associated stress. This requires the control of the
anti-symmetric part of the equilibrated stress for the use in an associated a posteriori error estimator.
One could perform the stress reconstruction in one of the available symmetric $H (\div)$-conforming stress spaces
like those introduced by Arnold and Winther \cite{ArnWin:02} or other ones included in the comparison \cite{CarEigGed:11}
(see \cite{NicWitWoh:08} or \cite{AinRan:10} for such approaches). 
But this complicates the stress reconstruction procedure significantly compared to the Raviart-Thomas elements (of degree $k$)
used here. This is particularly true in three dimensions where the lowest-order member of the symmetric $H (\div)$-conforming
finite element space constructed in \cite{ArnAwaWin:08} already involves polynomials of degree 4 and possesses 162 degrees
of freedom per tetrahedron. Equilibrated stress
reconstructions with weak symmetry are also considered in \cite{Kim:11a}, \cite{AinRan:11}, \cite{RieDiPErn:17}. These
approaches utilize special stress finite element spaces and are therefore less general than the one presented in this work.

The construction of equilibrated fluxes in broken Raviart-Thomas spaces is described in detail in \cite{BraPilSch:09} and
\cite{BraSch:08}.  More generally, a posteriori error estimation based on stress reconstruction has a long history with ideas 
dating back at least as far as \cite{LadLeg:83} and
\cite{PraSyn:47}. Recently, a unified framework for a posteriori error estimation based on stress reconstruction for the Stokes
system was carried out in \cite{HanSteVoh:12} (see also \cite{ErnVoh:15} for polynomial-degree robust estimates). These
two references include the treatment of nonconforming methods and both of them contain a historical perspective with a long list
of relevant references. Our weakly symmetric stress equilibration procedure is generalized to nonlinear elasticity associated with
a hyperelastic material model in \cite{BerMolSta:19}.



The outline of this paper is as follows. The next section starts by reviewing the displacement-pressure formulation for linear
elasticity and its approximation using the Taylor-Hood finite element pair. It then derives the conditions for a weakly symmetric
stress equilibration. The localization of the stress equilibration procedure is presented in Section \ref{sec-local_equilibration}.
Section \ref{sec-solvability} is concerned with the well-posedness of the local problems arising in the stress equilibration
procedure. In Section \ref{sec-estimates_asym_volum}, local upper estimates for the anti-symmetric
and volumetric stress components are provided which are crucial for the control of the constants associated with the reliability of
the a posteriori error estimates. Based on this, our a posteriori error estimator is derived first for the incompressible limit
case in Section \ref{sec-error_estimation_incompressible}. The effect of the data approximation is studied in detail in
Section \ref{sec-data_approximation}. Section \ref{sec-error_estimation_general} is then concerned with the
a posteriori error estimator for the general case. In Section \ref{sec-local_efficiency}, an upper bound by an appropriate residual
error estimator is established which leads to a local efficiency result for our weakly symmetric stress equilibration error estimator.
Finally, Section \ref{sec-numerical} shows numerical results for the popular Cook's membrane test problem which confirm the
theoretical predictions.

\section{Displacement-pressure formulation for incompressible linear elasticity and weakly symmetric stress reconstruction}

\label{sec-linear_elasticity_stress}

On a bounded domain $\Omega \subset \R^d$, $d = 2$ or $3$, assumed to be polygonally bounded such that the union of
elements in the triangulation $\cT_h$ coincides with $\Omega$, the boundary is split into $\Gamma_D$ (of positive surface
measure) and $\Gamma_N = \partial \Omega \backslash \Gamma_D$. We also assume that the families of triangulations
$\{ \cT_h \}$ are shape-regular and denote the diameter of an element $T \in \cT_h$ by $h_T$. The boundary value problem of
(possibly) incompressible linear elasticity consists in the saddle-point problem of finding $\bu \in H_{\Gamma_D}^1 (\Omega)^d$
and $p \in L^2 (\Omega)$ such that
\begin{equation}
  \begin{split}
    2 \mu ( \bepsilon (\bu) , \bepsilon (\bv) )_{L^2 (\Omega)} + ( p , \div \: \bv )_{L^2 (\Omega)} 
    & = ( \bff , \bv )_{L^2 (\Omega)} + \langle \bg , \bv \rangle_{L^2 (\Gamma_N)} \: ,\\
    ( \div \: \bu , q )_{L^2 (\Omega)} - \frac{1}{\lambda} ( p , q )_{L^2 (\Omega)} & = 0
  \end{split}
  \label{eq:disp_pressure}
\end{equation}
holds for all $\bv \in H_{\Gamma_D}^1 (\Omega)^d$ and $q \in L^2 (\Omega)$. Here, $\bff \in L^2 (\Omega)^d$ and
$\bg \in L^2 (\Gamma_N)^d$ are prescribed volume and surface traction forces, respectively. For the Lam\'e parameters,
$\mu$ is assumed to be on the order of one while $\lambda$ may become arbitrarily large modelling nearly incompressible
material behavior.
From now on, we will abbreviate the inner product in $L^2 (\omega)$ for some subset $\omega \subseteq \Omega$ by
$( \: \cdot \: , \: \cdot \: )_\omega$ (and simply write $( \: \cdot \: , \: \cdot \: )$ in the case of the entire domain $\omega = \Omega$).
For the $L^2 (\Gamma)$ inner product on a part of the boundary $\gamma \subseteq \partial \Omega$ we use the short-hand
notation $\langle \: \cdot \: , \: \cdot \: \rangle_{\gamma}$.
With respect to a suitable pair of finite element spaces $\bV_h \times Q_h$ representing
$H_{\Gamma_D}^1 (\Omega)^d \times L^2 (\Omega)$, the resulting finite-dimensional saddle-point problem consists in finding
$\bu_h \in \bV_h$ and $p_h \in Q_h$ such that
\begin{equation}
  \begin{split}
    2 \mu ( \bepsilon (\bu_h) , \bepsilon (\bv_h) ) + ( p_h , \div \: \bv_h ) &
    = ( \bff , \bv_h ) + \langle \bg , \bv_h \rangle_{\Gamma_N} \: , \\
    ( \div \: \bu_h , q_h ) - \frac{1}{\lambda} ( p_h , q_h ) & = 0
  \end{split}
  \label{eq:disp_pressure_discrete}
\end{equation}
holds for all $\bv_h \in \bV_h$ and $q_h \in Q_h$.
One possibility for the choice of the finite element spaces is, for $k \geq 1$, the Taylor-Hood pair consisting of
continuous piecewise polynomials of degree $k+1$ for each component of $\bV_h$
combined with continuous piecewise polynomials of
degree $k$ for $Q_h$. Our focus in this work is on that finite element combination but much of the derivation is also valid
for more general approaches.

The approximation
\begin{equation}
  \bsigma_h (\bu_h,p_h) = 2 \mu \bepsilon (\bu_h) + p_h \bI
  \label{eq:direct_stress_processing}
\end{equation}
which is obtained from the solution
$( \bu_h , p_h ) \in \bV_h \times Q_h$ of the discrete saddle point problem (\ref{eq:disp_pressure_discrete}) is,
in general, discontinuous and piecewise polynomial of degree $k$.
From $\bsigma_h (\bu_h,p_h)$, we reconstruct an $H (\div)$-conforming stress tensor $\bsigma_h^R$ in the
Raviart-Thomas space (componentwise) $\bSigma_h^R$ of order $k$, usually denoted by $RT_k^d$ (see, e.g.,
\cite[Sect. 2.3.1]{BofBreFor:13}). For the detailed definition of our stress reconstruction algorithm, we will also need the
broken Raviart-Thomas space
\begin{equation}
  \bSigma_h^{\Delta} = \{ \btau_h \in L^2(\Omega) : \left. \btau_h \right|_T \in RT_k (T)^d \} \: .
\end{equation}
By $\cS_h$ we denote the set of all sides (edges in 2D and faces in 3D) of the triangulation $\cT_h$.
For each $\bsigma_h^\Delta \in \bSigma_h^\Delta$ and each interior side $S \in \cS_h$, we define the jump
\begin{equation}
  \llbracket \bsigma_h^\Delta \cdot \bn \rrbracket_S
  = \left. \bsigma_h^\Delta \cdot \bn \right|_{T_-} - \left. \bsigma_h^\Delta \cdot \bn \right|_{T_+} \: ,
  \label{eq:definition_jump}
\end{equation}
where $\bn$ is the normal direction associated with $S$ (depending on its orientation) and $T_+$ and $T_-$ are the elements
adjacent to $S$ (such that $\bn$ points into $T_+$).
For sides $S \subset \Gamma_N$ located on the Neumann boundary, the jump in (\ref{eq:definition_jump}) is to be
interpreted as
\[
  \llbracket \bsigma_h^\Delta \cdot \bn \rrbracket_S = \left. \bsigma_h^\Delta \cdot \bn \right|_{T_-} \: ,
\]
assuming that $\bn$ points outside of $\Omega$. Moreover, a second type of jump is needed which we define as
\begin{equation}
  \llbracket \bsigma_h^\Delta \cdot \bn \rrbracket_S^\ast = \left\{ \begin{array}{lcr}
    \left. \bsigma_h^\Delta \cdot \bn \right|_{T_-} - \bg & , \mbox{ if } & S \subset \Gamma_N \: , \\
    \llbracket \bsigma_h^\Delta \cdot \bn \rrbracket_S & , \mbox{ if } & S \nsubseteq \Gamma_N \: .
  \end{array} \right.
  \label{eq:definition_jump_N}
\end{equation}
The introduction of the auxiliary type of jump in (\ref{eq:definition_jump_N}) allows us later to use the same formulas also for
patches adjacent to the Neumann boundary $\Gamma_N$.

We further define $\bZ_h$ as the space of
discontinuous $d$-dimensional vector functions which are piecewise polynomial of degree $k$. Similarly, $\bX_h$ stands for the
continuous $d (d-1)/2$-dimensional vector functions which are piecewise polynomial of degree $k$. For every
$d (d-1)/2$-dimensional vector $\btheta$ we define $\bJ^d (\btheta)$ by 
\begin{equation}
  \bJ^2(\theta) := \begin{pmatrix} 0 & \theta \\ -\theta & 0 \end{pmatrix}, 
  \quad \quad \bJ^3(\btheta) :=
  \begin{pmatrix} 0 & \theta_3 & -\theta_2 \\ -\theta_3 & 0 & \theta_1 \\ \theta_2 & -\theta_1 & 0 \end{pmatrix}
\label{eq:skew_symmetric_tensor}
\end{equation}
(cf. \cite[Sect. 9.3]{BofBreFor:13}). Finally, the broken inner product 
\begin{equation}
 ( \: \cdot \: , \: \cdot \: )_h := \sum_{T \in \cT_h} ( \: \cdot \: , \: \cdot \: )_T \: ,
 \label{eq:broken_inner_product}
\end{equation}
will be used, where $( \: \cdot \: , \: \cdot \: )_T$ is the $L^2(T)$ inner product. 

We follow the general idea of equilibration (cf. \cite[Sect. III.9]{Bra:07}, \cite{BraSch:08}) and extend it to the
case of weakly symmetric stresses. The construction is done for the difference
$\bsigma_h^\Delta := \bsigma_h^R - \bsigma_h (\bu_h , p_h)$ 
between the reconstructed and the original stress, which is an element of $\bSigma_h^{\Delta}$.
In order to correspond
to an admissible stress reconstruction $\bsigma_h^R$, the following conditions need to be satisfied for $\bsigma_h^\Delta$:
\begin{equation}
  \begin{split}
    ( \div \: \bsigma_h^\Delta , \bz_h )_{h}
    & = - ( \bff + \div \: \bsigma_h (\bu_h,p_h) , \bz_h )_{h} \mbox{ for all } \bz_h \in \bZ_h \: , \\
    \langle \llbracket \bsigma_h^\Delta \cdot \bn \rrbracket_S , \bzeta \rangle_S
    & = - \langle \llbracket \bsigma_h (\bu_h,p_h) \cdot \bn \rrbracket_S^\ast , \bzeta \rangle_S \mbox{ for all }
    \bzeta \in P_k (S)^d \: , \: S \in \cS_h^\ast \: , \\
    ( \bsigma_h^\Delta , \bJ^d (\bgamma_h) ) & = 0 \mbox{ for all } \bgamma_h \in \bX_h \: 
  \end{split}
  \label{eq:equilibration_conditions}
\end{equation}
where $\cS_{h}^\ast := \{ S \in \cS_h: S \nsubseteq \Gamma_D \}$.
Due to our specific choice of $\bZ_h$, the first equation in (\ref{eq:equilibration_conditions}) implies that, on
each $T \in \cT_h$, $\div \: \bsigma_h^\Delta = - \cP_h^k \bff - \div \: \bsigma_h (\bu_h,p_h)$ holds, where $\cP_h^k$ denotes the
element-wise $L^2$ projection onto the space of polynomials of degree $k$. Moreover, on sides located
on the Neumann boundary $\Gamma_N$, (\ref{eq:definition_jump}) and (\ref{eq:definition_jump_N}) lead to
$\bsigma_h^\Delta \cdot \bn = \cP_{h,\Gamma}^k \bg - \bsigma_h (\bu_h,p_h) \cdot \bn$, where $\cP_{h,\Gamma}^k$ denotes the
side-wise $L^2$ projection onto the polynomials of degree $k$.

\section{Local stress equilibration procedure}

\label{sec-local_equilibration}

For the purpose of localizing the reconstruction and deriving local efficiency bounds we make use of a partition of unity. The
commonly used partition of unity with respect to the set $\cV_h$ of all vertices of $\cT_h$,
\begin{equation}
  1 \equiv \sum_{z \in \cV_h} \phi_z \mbox{ on } \Omega \: ,
  \label{eq:partition_of_unity_1}
\end{equation}
consists of continuous piecewise linear functions $\phi_z$. In this case, the support of $\phi_z$ is restricted to
\begin{equation}
  \omega_z := \bigcup \{ T \in \cT_h : z \mbox{ is a vertex of } T \} \: .
  \label{eq:vertex_patch}
\end{equation}
For reasons which will be explained further below in this section, the classical partition of unity has to be modified in order to
exclude patches formed by vertices $z \in \Gamma_N$. To this end, let $\cV_h^\ast = \{ z \in \cV_h : z \notin \Gamma_N \}$
denote the subset of vertices which are not located on a side (edge/face) of $\Gamma_N$. The modified partition of unity is
defined by
\begin{equation}
  1 \equiv \sum_{z \in \cV_h^\ast} \phi_z^\ast \mbox{ on } \Omega \: .
  \label{eq:partition_of_unity}
\end{equation}
For $z \in \cV_h^\ast$ not connected by an edge to $\Gamma_N$ the function $\phi_z^\ast$ is equal to $\phi_z$. Otherwise,
the function $\phi_z^\ast$ has to be modified in order to account for unity at the connected vertices on $\Gamma_N$. For each
$z_N \in \Gamma_N$ one vertex $z_I \notin \Gamma_N$ connected by an edge with $z_N$ is chosen and $\phi_{z_I}$ is
extended by the value $1$ along the edge from $z_I$ to $z_N$ to obtain the modified function $\phi_{z_I}^\ast$. The support of
$\phi_z^\ast$ is denoted by
\begin{equation}
  \omega_z^\ast := \bigcup \{ T \in \cT_h : \phi_z^\ast = 1 \mbox{ for at least one vertex } \hat{z} \mbox{ of } T \} \: .
  \label{eq:vertex_patch_prime}
\end{equation}
For the partition of unity (\ref{eq:partition_of_unity}) to hold, we require the triangulation $\cT_h$ to be such that each vertex on
$\Gamma_N$ is connected to an interior edge.
For the localization of the reconstruction algorithm, we will also need the local subspaces
\begin{equation}
  \begin{split}
    \bSigma_{h,z}^{\Delta} & = \{ \btau_h \in \bSigma_h^{\Delta} : \btau_h \cdot \bn = \bzero \mbox{ on }
    \partial \omega_z^\ast \backslash \partial \Omega \: , \:
    \btau_h \equiv \bzero \mbox{ on } \Omega \backslash \omega_z^\ast \} \: , \\
    \bZ_{h,z} & = \{ \left. \bz_h \right|_{\omega_z^\ast} : \bz_h \in \bZ_h \} \: , \\
    \bX_{h,z} & = \{ \left. \bgamma_h \right|_{\omega_z^\ast} : \bgamma_h \in \bX_h \} \: ,
  \end{split}
 \label{eq:local_subspaces}
\end{equation}
as well as the local sets of sides $\cS_{h,z}^\ast := \lbrace S \in \cS_h^\ast : S \subset \overline{\omega}_z^\ast \rbrace$.

The conditions in (\ref{eq:equilibration_conditions}) can be satisfied by a sum of patch-wise contributions
\begin{equation}
  \bsigma_h^\Delta = \sum_{z \in \cV_h^\ast} \bsigma_{h,z}^\Delta \: ,
  \label{eq:patch_decomposition}
\end{equation}
where, for each $z \in \cV_h^\ast$, $\bsigma_{h,z}^\Delta \in  \bSigma_{h,z}^{\Delta}$ is computed such that
$\Vert \bsigma_{h,z}^\Delta \Vert_{\omega_z^\ast}^2 $ is minimized subject to the following constraints:
\begin{equation}
  \begin{split}
    ( \div \: \bsigma_{h,z}^\Delta , \bz_{h,z} )_{\omega_z^\ast,h}
    & = - ( ( \bff + \div \: \bsigma_h (\bu_h,p_h) ) \phi_z^\ast , \bz_{h,z}  )_{\omega_z^\ast,h}
    \mbox{ for all } \bz_{h,z} \in \bZ_{h,z} \: , \\
    \langle \llbracket \bsigma_{h,z}^\Delta \cdot \bn \rrbracket_S , \bzeta \rangle_S
    = & - \langle \llbracket \bsigma_h (\bu_h,p_h) \cdot \bn \rrbracket_S^\ast \, \phi_z^\ast , \bzeta \rangle_S
    \mbox{ for all } \bzeta \in P_k (S)^d \: , \: S \in \cS_{h,z}^\ast \: , \\
    ( \bsigma_{h,z}^\Delta , \bJ^d (\bgamma_{h,z}) )_{\omega_z^\ast} & = 0
    \hspace{3.7cm} \mbox{ for all } \bgamma_{h,z} \in \bX_{h,z} \: .
  \end{split}
  \label{eq:equilibration_conditions_local}
\end{equation}
For each $z \in \cV_h^\ast$, this is a linearly-constrained quadratic minimization problem of low dimension.
In a similar way as in \cite{CaiZha:12a}, it can be solved in the following two substeps using the subspace
\begin{equation}
  \bSigma_{h,z}^{\Delta,\div} :=  \lbrace \btau_h \in \bSigma_{h,z}^\Delta : \llbracket \btau_h \cdot \bn \rrbracket_S = \bzero
  \mbox{ for all } S \in \cS_{h,z}^\ast \: , \: \div \: \btau_h = \bzero \rbrace \: :
\end{equation}
\textit{Step 1:} Compute an arbitrary $\bsigma_{h,z}^{\Delta,1} \in   \bSigma_{h,z}^{\Delta}$ satisfying the first two equalities in (\ref{eq:equilibration_conditions_local}).\\
\textit{Step 2:} Compute $\bsigma_{h,z}^{\Delta,2} \in  \bSigma_{h,z}^{\Delta,\div}$  such that
$\Vert \bsigma_{h,z}^{\Delta,1} + \bsigma_{h,z}^{\Delta,2} \Vert_{\omega_z^\ast}^2$ is minimized and
\begin{equation}
( \bsigma_h^{\Delta,2} , \bJ^d (\bgamma_h) )_{\omega_z^\ast}  = - ( \bsigma_h^{\Delta,1} , \bJ^d (\bgamma_h) )_{\omega_z^\ast}
\mbox{ for all } \bgamma_h \in \bX_{h,z}
\end{equation}
is satisfied. Finally, set $\bsigma_{h,z}^\Delta = \bsigma_{h,z}^{\Delta,1} + \bsigma_{h,z}^{\Delta,2}$. 

For the computation of $\bsigma_{h,z}^{\Delta,1}$ in \textit{Step 1}, the explicit formulas from \cite{CaiZha:12a} can be used.
The remaining minimization problem in \textit{Step 2} is of much smaller size than for the original problem
(\ref{eq:equilibration_conditions_local}).

We remark that the modification of the partition of unity (\ref{eq:partition_of_unity}) is only necessary in the two-dimensional case
and even then it can be avoided if the triangulation is such that each vertex $z_N \in \Gamma_N$ is connected to at least two
edges which are not part of $\Gamma_N$. However, using the standard partition of unity without this mesh property will (in 2D)
lead to patches $\omega_z$ around vertices on $\Gamma_N$ consisting of only two triangles. 
For those patches the local space $\bSigma^\Delta_{h,z}$ does not exhibit enough degrees of freedom to satisfy all equations in
(\ref{eq:equilibration_conditions_local}) unless $\partial\omega_z \cap \Gamma_D \neq \emptyset$. In the three-dimensional case
it is sufficient for each vertex $z_N \in \Gamma_N$ to be connected to one interior edge.

\section{Well-posedness of the local problems on vertex patches}

\label{sec-solvability}

The local minimization problem subject to the constraints (\ref{eq:equilibration_conditions_local}) can be guaranteed to possess a
unique solution if, for every right hand side, a function $\bsigma_{h,z}^\Delta \in \bSigma_{h,z}^\Delta$ exists such that the
constraints (\ref{eq:equilibration_conditions_local}) are satisfied. To this end, the range of the linear operator on the left-hand
side of (\ref{eq:equilibration_conditions_local}) is of interest.

\begin{proposition}
  The subspace
  \begin{equation}
    \begin{split}
      \bR_{h,z}^\perp := & \{ ( \bz_{h,z} , \bgamma_{h,z} ) \in \bZ_{h,z} \times \bX_{h,z} :
      \exists \bzeta_S \in P_k (S)^d \: , \: S \in \cS_{h,z}^\ast \mbox{ such that } \\
      & ( \div \: \bsigma_{h,z}^\Delta , \bz_{h,z} )_{\omega_z^\ast,h}
      - \sum_{S \in \cS_{z,h}^\ast} \langle \llbracket \bsigma_{h,z}^\Delta \cdot \bn \rrbracket_S , \bzeta_S \rangle_S
      + ( \bsigma_{h,z}^\Delta , \bJ^d (\bgamma_{h,z}) )_{\omega_z^\ast} = 0 \\
      & \hspace{6cm} \mbox{ holds for all } \bsigma_{h,z}^\Delta \in \bSigma_{h,z}^\Delta \} \: ,
    \end{split}
    \label{eq:adjoint_null_space}
  \end{equation}
  i.e., the null space of the adjoint operator associated with the constraints (\ref{eq:equilibration_conditions_local}), can be
  characterized as follows:
  \begin{equation}
    \begin{split}
      \bR_{h,z}^\perp & = \{ ( \bzero , \bzero ) \} \mbox{ if } \partial \omega_z^\ast \cap \Gamma_D \neq \emptyset \: , \\
      \bR_{h,z}^\perp & = \{ ( \brho , \btheta ) \in \bR\bM \times \R^{d (d-1)/2} : \bJ^d (\btheta) = \bas \: \bnabla \brho \}
      \mbox{ if } \partial \omega_z^\ast \cap \Gamma_D = \emptyset \: ,
    \end{split}
    \label{eq:adjoint_null_space_characterization}
  \end{equation}
  where $\bR\bM = \{ \brho : \omega_z^\ast \rightarrow \R^d : \bepsilon (\brho) = \bzero \}$ denotes the space of rigid body modes
  and $\bas \: \btau = (\btau - \btau^T)/2$ stands for the anti-symmetric part of a function
  $\btau :  \Omega \rightarrow \R^{d \times d}$.
  \label{prop-adjoint_null_space_characterization}
\end{proposition}

\begin{proof}
  If we restrict ourselves to $\bsigma_{h,z}^\Delta \in \bSigma_{h,z}^\Delta$ with
  $\llbracket \bsigma_{h,z}^\Delta \cdot \bn \rrbracket_S = 0$ for all $S \in \omega_z^\ast$, then we end up with the
  $H (\div)$-conforming Raviart-Thomas space $RT_k^d$. The condition in (\ref{eq:adjoint_null_space}) for the definition of
  $\bR_{h,z}^\perp$ simplifies to
  \begin{equation}
    ( \div \: \bsigma_{h,z}^\Delta , \bz_{h,z} )_{\omega_z^\ast,h}
    + ( \bsigma_{h,z}^\Delta , \bJ^d (\bgamma_{h,z}) )_{\omega_z^\ast} = 0 \: .
    \label{eq:condition_conforming}
  \end{equation}
  The inf-sup stability of the finite element combination $RT_k^d$ (for the stress) with $\bZ_{h,z} \times \bX_{h,z}$ (for the
  displacement and rotation), shown in \cite{BofBreFor:09}, implies that $\bR_{h,z}^\perp$ is contained in the null space of
  the continuous problem given by (\ref{eq:adjoint_null_space_characterization}).
  
  On the other hand, $\bR_{h,z}^\perp$ does indeed contain all the functions given in (\ref{eq:adjoint_null_space_characterization})
  since, setting $( \bz_{h,z} , \bgamma_{h,z} ) = ( \brho , \btheta )$ with $\bJ^d (\btheta) = \bas \: \bnabla \brho$ and
  $\bzeta_S = \left. \brho \right|_S$, we have, for all $\bsigma_{h,z}^\Delta \in \bSigma_{h,z}^\Delta$, that
  \begin{equation}
    \begin{split}
      & ( \div \: \bsigma_{h,z}^\Delta , \bz_{h,z} )_{\omega_z^\ast,h}
      - \sum_{S \in \cS_{h,z}^\ast} \langle \llbracket \bsigma_{h,z}^\Delta \cdot \bn \rrbracket_S , \bzeta_S \rangle_S
      + ( \bsigma_{h,z}^\Delta , \bJ^d (\bgamma_{h,z}) )_{\omega_z^\ast} \\
      & = ( \div \: \bsigma_{h,z}^\Delta , \brho )_{\omega_z^\ast,h}
      - \sum_{S \in \cS_{h,z}^\ast} \langle \llbracket \bsigma_{h,z}^\Delta \cdot \bn \rrbracket_S , \brho \rangle_S
      + ( \bsigma_{h,z}^\Delta , \bJ^d (\btheta) )_{\omega_z^\ast} \\
      & = - ( \bsigma_{h,z}^\Delta , \bnabla \brho )_{\omega_z^\ast}
      + ( \bsigma_{h,z}^\Delta , \bas \: \bnabla \brho )_{\omega_z^\ast}
      = - ( \bsigma_{h,z}^\Delta , \bepsilon (\brho) )_{\omega_z^\ast} = 0
    \end{split}
  \end{equation}
  holds (note that $\bepsilon (\brho) = \bzero$ for $\brho \in \bR\bM$).
\end{proof}

Proposition \ref{prop-adjoint_null_space_characterization} will now be used in order to show that it is possible to satisfy the
constraints in (\ref{eq:equilibration_conditions_local}). For vertices $z \in \cV_h^\ast$ with
$\partial \omega_z^\ast \cap \Gamma_D \neq \emptyset$, there is no restriction on the right-hand side in
(\ref{eq:equilibration_conditions_local}) and there will always be a unique solution. However, if
$\partial \omega_z^\prime \cap \Gamma_D = \emptyset$, the range of the
left-hand side operator does not cover the full space and therefore a compatibility condition needs to be fulfilled by the
the right-hand side in (\ref{eq:equilibration_conditions_local}). More precisely, the right-hand side has to be perpendicular to
$\bR_{h,z}^\perp$ which, in view of Proposition \ref{prop-adjoint_null_space_characterization}, means that
\begin{equation}
  ( ( \bff + \div \: \bsigma_h (\bu_h,p_h) ) \phi_z^\ast , \brho  )_{\omega_z^\ast,h}
  = \sum_{S \in \cS_{h,z}^\ast}
  \langle \llbracket \bsigma_h (\bu_h,p_h) \cdot \bn \rrbracket_S^\ast \, \phi_z^\ast , \brho \rangle_S
  \label{eq:compatibility_conditions}
\end{equation}
has to hold for all $( \brho , \btheta ) \in \bR\bM \times \R^{d (d-1) / 2}$ with $\bJ^d (\btheta) = \bas \: \bnabla \brho$. That this is
indeed true can be seen as follows: The first term in (\ref{eq:compatibility_conditions}) can be rewritten as
\begin{equation}
  \begin{split}
    ( ( \bff & + \div \: \bsigma_h (\bu_h,p_h) ) \phi_z^\ast , \brho  )_{\omega_z^\ast,h}
    = ( \bff + \div \: \bsigma_h (\bu_h,p_h) , \phi_z^\ast  \brho  )_{\omega_z^\ast,h} \\
    & = ( \bff , \phi_z^\ast \brho  )_{\omega_z^\ast}
    + \sum_{S \in \cS_{h,z}^\ast}
    \langle \llbracket \bsigma_h (\bu_h,p_h) \cdot \bn \rrbracket_S , \phi_z^\ast \brho \rangle_S
    - ( \bsigma_h (\bu_h,p_h) , \bnabla (\phi_z^\ast \brho) )_{\omega_z^\ast}
  \end{split}
\end{equation}
by partial integration. Using the fact that $\llbracket \: \cdot \: \rrbracket_S$ and $\llbracket \: \cdot \: \rrbracket_S^\ast$ differ only
on sides $S \subset \Gamma_N$ and recalling that $\bsigma_h (\bu_h,p_h)$ is symmetric, we end up with
\begin{equation}
  \begin{split}
    ( ( \bff + \div \: & \bsigma_h (\bu_h,p_h) ) \phi_z^\ast , \brho  )_{\omega_z^\ast,h}
    = ( \bff , \phi_z^\ast \brho  )_{\omega_z^\ast} + \langle \bg , \phi_z^\ast \brho \rangle_{\Gamma_N} \\
    & + \sum_{S \in \cS_{h,z}^\ast}
    \langle \llbracket \bsigma_h (\bu_h,p_h) \cdot \bn \rrbracket_S^\ast , \phi_z^\ast \brho \rangle_S
    - ( \bsigma_h (\bu_h,p_h) , \bepsilon (\phi_z^\ast \brho) )_{\omega_z^\ast} \: .
  \end{split}
  \label{eq:compatibility_conditions_final}
\end{equation}
Using the fact that $\phi_z^\ast \brho \in \bV_h$, the first equation in (\ref{eq:disp_pressure_discrete}) leads to
(\ref{eq:compatibility_conditions}).

\section{Vertex-patch estimates for the anti-symmetric and volumetric stress errors}

\label{sec-estimates_asym_volum}

This section provides upper bounds for two terms that will arise later in the derivation of the error estimators. These terms
involve the anti-symmetric and deviatoric stress parts and are crucial for the treatment of linear elasticity with
guaranteed upper bound which are only dependent on the shape of the triangulation and not on the considered problem,
i.e., the location and type of the boundary conditions.
For $\btau : \Omega \rightarrow \R^{d \times d}$, let us denote by
$\bdev \: \btau = \btau - (\tr \: \btau) \bI / d$ the deviatoric, i.e. trace-free, part.

\begin{lemma}
  Let $( \bu_h , p_h ) \in \bV_h \times Q_h$ be the solution of (\ref{eq:disp_pressure_discrete}) and let
  $\bsigma_h^R \in \bSigma_h^R$ be a stress reconstruction satisfying the weak symmetry condition
  $( \bsigma_h^R , \bJ^d (\bgamma_h) ) = 0$ for all $\bgamma_h \in \bX_h$. Then,
  \begin{equation}
    \left| ( \bas \: \bsigma_h^R , \bnabla ( \bu - \bu_h ) ) \right|
    \leq C_K \| \bas \: \bsigma_h^R \| \: \| \bepsilon ( \bu - \bu_h ) \|
    \label{eq:bound_asym}
  \end{equation}
  holds with a constant $C_K$ which depends only on (the largest interior angle in) the triangulation $\cT_h$.
  
  Moreover, if $Q_h$ is such that it contains the space of piecewise linear continuous functions, then
  \begin{equation}
    \left| ( \tr ( \bsigma - \bsigma_h^R ) , \div \: \bu_h - \frac{1}{\lambda} p_h ) \right|
    \leq C_A \| \bdev ( \bsigma - \bsigma_h^R ) \| \: \| \div \: \bu_h - \frac{1}{\lambda} p_h \| \: ,
    \label{eq:bound_volum}
  \end{equation}
  where, again, $C_A$ depends only on (the largest interior angle in) the triangulation $\cT_h$.
  \label{lemma-bounds_asym_volum}
\end{lemma}

\begin{proof}
  For both inequalities (\ref{eq:bound_asym}) and (\ref{eq:bound_volum}), the (standard) partition of unity
  \begin{equation}
    1 \equiv \sum_{z \in \cV_h} \phi_z \mbox{ on } \Omega
    \label{eq:partition_all}
  \end{equation}
  with respect to the set of all vertices in the triangulation $\cV_h$ is used. For proving (\ref{eq:bound_asym}), the weak symmetry
  property of the stress reconstruction $\bsigma_h^R$ implies
  \begin{equation}
    ( \bas \: \bsigma_h^R , \bnabla (\bu - \bu_h) ) = ( \bas \: \bsigma_h^R , \bnabla (\bu - \bu_h) - \bJ^d (\balpha_h) )
    \mbox{ for all } \balpha_h = \sum_{z \in \cV_h} \balpha_z \phi_z
  \end{equation}
  with $\balpha_z \in \R^{d (d-1)/2}$.
  Using (\ref{eq:partition_all}) we are led to
  \begin{equation}
    \begin{split}
      | ( \bas \: \bsigma_h^R , \bnabla (\bu - \bu_h) ) |
      & = | ( \bas \: \bsigma_h^R , \sum_{z \in \cV_h} \left( \bnabla (\bu - \bu_h) - \bJ^d (\balpha_z) \right) \phi_z ) | \\
      & = \left| \sum_{z \in \cV_h} ( \bas \: \bsigma_h^R ,
      \left( \bnabla (\bu - \bu_h) - \bJ^d (\balpha_z) \right) \phi_z )_{\omega_z} \right| \\
      & = \left| \sum_{z \in \cV_h} ( (\bas \: \bsigma_h^R) \phi_z , \bnabla (\bu - \bu_h) - \bJ^d (\balpha_z) )_{\omega_z} \right| \\
      & \leq \sum_{z \in \cV_h} \| (\bas \: \bsigma_h^R) \phi_z \|_{\omega_z}
      \| \bnabla (\bu - \bu_h) - \bJ^d (\balpha_z) \|_{\omega_z} \\
      & \leq \sum_{z \in \cV_h} \| \bas \: \bsigma_h^R \|_{\omega_z} \| \bnabla (\bu - \bu_h) - \bJ^d (\balpha_z) \|_{\omega_z} \: .
    \end{split}
    \label{eq:asym_partition}
  \end{equation}
  For all rigid body modes $\brho \in \bR\bM$, $\bnabla \brho = \bJ^d (\balpha_z)$ holds with some $\balpha_z \in \R^{d (d-1)/2}$
  and therefore
  \begin{equation}
    \inf_{\balpha_z} \| \bnabla (\bu - \bu_h) - \bJ^d (\balpha_z) \|_{\omega_z}
    \leq \inf_{\brho \in \bR\bM} \| \bnabla (\bu - \bu_h - \brho) \|_{\omega_z}
    \leq C_{K,z} \| \bepsilon (\bu - \bu_h) \|_{\omega_z}
    \label{eq:Korn_average}
  \end{equation}
  due to Korn's inequality (cf. \cite{Hor:95}). The constant $C_{K,z}$ obviously only depends on the geometry of the vertex patch
  $\omega_z$ or, more precisely, on its largest interior angle. If we define $C_K = (d+1) \max \{ C_{K,z} : z \in \cV_h \}$, we finally
  obtain from (\ref{eq:asym_partition}) that
  \begin{equation}
    \begin{split}
      | ( \bas \: \bsigma_h^R , & \bnabla (\bu - \bu_h) ) |
      \leq \frac{C_K}{d+1} \sum_{z \in \cV_h} \| \bas \: \bsigma_h^R \|_{\omega_z} \| \bepsilon (\bu - \bu_h) \|_{\omega_z} \\
      & \leq C_K \left( \frac{1}{d+1} \sum_{z \in \cV_h} \| \bas \: \bsigma_h^R \|_{\omega_z}^2 \right)^{1/2}
      \left( \frac{1}{d+1} \sum_{z \in \cV_h} \| \bepsilon (\bu - \bu_h) \|_{\omega_z}^2 \right)^{1/2} \\
      & = C_K \| \bas \: \bsigma_h^R \| \: \| \bepsilon (\bu - \bu_h) \|
    \end{split}
  \end{equation}
  holds, where we used the fact that each element (triangle or tetrahedron) is contained in exactly $d+1$ vertex patches.
  
  For proving (\ref{eq:bound_volum}), we observe that the second equation in (\ref{eq:disp_pressure_discrete}) together
  with our assumption on $Q_h$ implies
  \begin{equation}
    \begin{split}
      ( \tr (\bsigma - \bsigma_h^R) , \div \: \bu_h - \frac{1}{\lambda} p_h )
      & = ( \tr (\bsigma - \bsigma_h^R) - \beta_h , \div \: \bu_h - \frac{1}{\lambda} p_h ) \\
      & \mbox{ for all } \beta_h = \sum_{z \in \cV_h} \beta_z \phi_z \: , \: \beta_z \in \R \: .
    \end{split}
  \end{equation}
  Again using the partition of unity (\ref{eq:partition_all}), we obtain
  \begin{equation}
    \begin{split}
      ( \tr (\bsigma - \bsigma_h^R) , \div \: \bu_h - \frac{1}{\lambda} p_h )
      & = \sum_{z \in \cV_h}
      ( (\tr (\bsigma - \bsigma_h^R) - \beta_z) \phi_z , \div \: \bu_h - \frac{1}{\lambda} p_h )_{\omega_z} \\
      & = \sum_{z \in \cV_h}
      ( \tr (\bsigma - \bsigma_h^R) - \beta_z , (\div \: \bu_h - \frac{1}{\lambda} p_h) \phi_z )_{\omega_z} .
    \end{split}
    \label{eq:volum_partition}
  \end{equation}
  We choose $\beta_z$ in such a way that $( \tr (\bsigma - \bsigma_h^R) - \beta_z , 1 )_{\omega_z} = 0$
  and use the ``dev-div lemma'' (cf. \cite[Prop. 9.1.1]{BofBreFor:13})
  \begin{equation}
    \| \tr \: \btau - \beta_z \|_{\omega_z} \leq C_{A,z} \| \bdev  \: \btau \|_{\omega_z}
    \label{eq:dev_div_average}
  \end{equation}
  which holds for any $\tau \in H (\div,\omega_z)$ with $\div \: \btau = \bzero$. Since $\div (\bsigma - \bsigma_h^R) = \bzero$
  this leads to
  \begin{equation}
    \begin{split}
      \left| ( \tr (\bsigma - \bsigma_h^R) - \beta_z , \right.
      & \left. (\div \: \bu_h - \frac{1}{\lambda} p_h) \phi_z )_{\omega_z} \right| \\
      & \leq \| \tr (\bsigma - \bsigma_h^R) - \beta_z \|_{\omega_z} \| (\div \: \bu_h - \frac{1}{\lambda} p_h) \phi_z \|_{\omega_z} \\
      & \leq C_{A,z} \| \bdev (\bsigma - \bsigma_h^R) \|_{\omega_z} \| \div \: \bu_h - \frac{1}{\lambda} p_h \|_{\omega_z} \: ,
    \end{split}
  \end{equation}
  where $C_{A,z}$ depends only on the shape of $\omega_z$.
  Setting $C_A = (d+1) \max \{ C_{A,z} : z \in \cV_h \}$ and inserting this into (\ref{eq:volum_partition}) finally leads to
  \begin{equation}
    \begin{split}
      & | ( \tr (\bsigma - \bsigma_h^R) , \div \: \bu_h - \frac{1}{\lambda} p_h ) |
      \leq \sum_{z \in \cV_h} C_{A,z} \| \bdev (\bsigma - \bsigma_h^R) \|_{\omega_z}
      \| \div \: \bu_h - \frac{1}{\lambda} p_h \|_{\omega_z} \\
      & \leq C_A \left( \sum_{z \in \cV_h} \frac{1}{d+1} \| \bdev (\bsigma - \bsigma_h^R) \|_{\omega_z}^2 \right)^{1/2}
      \left( \sum_{z \in \cV_h} \frac{1}{d+1} \| \div \: \bu_h - \frac{1}{\lambda} p_h \|_{\omega_z}^2 \right)^{1/2} \\
      & = C_A \| \bdev (\bsigma - \bsigma_h^R) \| \: \| \div \: \bu_h - \frac{1}{\lambda} p_h \|
    \end{split}
  \end{equation}
  and concludes the proof.
\end{proof}

The constants $C_{K,z}$ from (\ref{eq:Korn_average}) corresponds to the second case in \cite{Hor:95} and are known to be not
smaller than 2 (which is the value for a perfect disc). In principle, upper bounds for $C_{K,z}$ can be computed for any vertex
patch using the formulas in \cite[Sect. 5]{Hor:95}. In particular, for a vertex patch $\omega_z$ consisting of six equilateral triangles,
we have $C_{K,z} \leq \sqrt{8}$ from \cite[(5.17)]{Hor:95}.

In the two-dimensional case, the constant $C_{A,z}$ from (\ref{eq:dev_div_average}) is related to $C_{K,z}$ by
\begin{equation}
  C_{A,z} \leq 2 \left( C_{K,z}^2 - 1 \right)^{1/2} \: ,
  \label{eq:relation_A_K}
\end{equation}
which can be seen as follows: Korn's inequality of the type (\ref{eq:Korn_average}) implies that
\begin{equation}
  \| \bepsilon (\bv) \|_{\omega_z}^2 + \inf_{\alpha \in \R} \| \bas \: \bnabla \bv - J^2 (\alpha) \|_{\omega_z}^2
  = \inf_{\alpha \in \R} \| \bnabla \bv - J^2 (\alpha) \|_{\omega_z}^2 \leq C_{K,z}^2 \| \bepsilon (\bv) \|_{\omega_z}^2
  \label{eq:Korn_implication}
\end{equation}
holds for all $\bv \in H^1 (\omega_z)^d$. Due to the special form of $\bas$ and $J^2 (\alpha)$ this may be rewritten as
\begin{equation}
  \frac{1}{2} \inf_{\alpha \in \R} \| \curl \: \bv - 2 \alpha \|_{\omega_z}^2 \leq \left( C_{K,z}^2 - 1 \right) \| \bepsilon (\bv) \|_{\omega_z}^2 \: .
  \label{eq:Korn_implication_2}
\end{equation}
In order to derive the desired inequality (\ref{eq:dev_div_average}) from this, the representation $\btau = \bnabla^\perp \bv$ with
$\bv \in H^1 (\omega_z)$, which holds due to $\div \: \btau = \bzero$, is used. This leads to
\begin{equation}
  \tr \: \btau = \curl \: \bv \mbox{ and } \bdev \: \btau = \bepsilon (\bv)
  \begin{pmatrix} 0 & -1 \\ 1 & \;\;0 \end{pmatrix} \: ,
\end{equation}
which implies that (\ref{eq:dev_div_average}) holds with $C_{A,z} = 2 \left( C_{K,z}^2 - 1 \right)^{1/2}$.

\section{A posteriori error estimation: Incompressible case}

\label{sec-error_estimation_incompressible}

In this section, our a posteriori error estimator based on the stress equilibration $\bsigma_h^\Delta$ is derived under
simplifying assumptions that make the analysis less complicated and clarifies the main ideas. To this end, we restrict
ouselves to the incompressible limit where $\lambda$ is set to infinity. Moreover, we assume that $\bff$ is piecewise
polynomial of degree $k$ with respect to $\cT_h$ and that $\bg$ is piecewise polynomial of degree $k$ with respect to
$\cS_h \cap \Gamma_N$ (implying that $\bff = \cP_h^k \bff$ and $\bg = \cP_{h,\Gamma}^k \bg$). The justification of
this assumption will be postponed to the next section. After that, Section \ref{sec-error_estimation_general} contains
the more technical analysis for arbitrary Lam\'e parameter $\lambda$.

Our aim is to estimate the displacement error with respect to $\| \bepsilon ( \: \cdot \: ) \|$ which constitutes a norm on
$H_{\Gamma_D}^1 (\Omega)^d$ due to Korn's inequality. The definition of the stress leads directly to
\begin{equation}
  \tr \: \bsigma = 2 \mu \: \div \: \bu + d \: p = d \: p \: , \:
  \tr \: \bsigma_h (\bu_h,p_h) = 2 \mu \: \div \: \bu_h + d \: p_h \: ,
\end{equation}
which implies
\begin{equation}
  \bepsilon (\bu) = \frac{1}{2 \mu} \left( \bsigma - p \bI \right)
  = \frac{1}{2 \mu} \left( \bsigma - \frac{1}{d} (\tr \: \bsigma) \bI \right)
  =: \cA_\infty \bsigma
  \label{eq:strain_stress_relation_inf}
\end{equation}
and
\begin{equation}
  \begin{split}
    \bepsilon (\bu_h) & = \frac{1}{2 \mu} \left( \bsigma_h - p_h \bI \right) \\
    & = \frac{1}{2 \mu} \left( \bsigma_h - \frac{1}{d} (\tr \: \bsigma_h) \bI \right) + \frac{1}{d} (\div \: \bu_h) \: \bI
    = \cA_\infty \bsigma_h + \frac{1}{d} \: (\div \: \bu_h) \: \bI \: .
  \end{split}
  \label{eq:strain_stress_relation_discrete_inf}
\end{equation}
Inserting the relation $\bsigma = 2 \mu \bepsilon (\bu) + p \bI$ which holds for the exact solution, we obtain
\begin{equation}
  \begin{split}
    \| \bsigma_h^\Delta \|_{\cA_\infty}^2 & = \| \bsigma_h^R - \bsigma_h (\bu_h,p_h) \|_{\cA_\infty}^2
    = \| \bsigma - \bsigma_h^R - 2 \mu \bepsilon (\bu - \bu_h) - (p - p_h) \bI \|_{\cA_\infty}^2 \\
    & = \| \bsigma - \bsigma_h^R \|_{\cA_\infty}^2 + \| 2 \mu \bepsilon (\bu - \bu_h) + (p - p_h) \bI  \|_{\cA_\infty}^2 \\
    & \;\;\; - 2 ( \bsigma - \bsigma_h^R , 2 \mu \bepsilon (\bu - \bu_h) + (p - p_h) \bI )_{\cA_\infty} \\
    & = \frac{1}{2 \mu} \| \bdev (\bsigma - \bsigma_h^R) \|^2
    + ( 2 \mu \bepsilon (\bu - \bu_h) + (p - p_h) \bI - 2 ( \bsigma - \bsigma_h^R ) , \\
    & \hspace{5.5cm} \cA_\infty (2 \mu \bepsilon (\bu - \bu_h) + (p - p_h) \bI) ) \: .
  \end{split}
  \label{eq:estimator_conforming_first}
\end{equation}
The right term in the last inner product can be rewritten as
\begin{equation}
  \begin{split}
    \cA_\infty ( 2 \mu \bepsilon (\bu - \bu_h) + (p - p_h) \bI )
    = \bepsilon (\bu - \bu_h) + \frac{1}{d} ( \div \: \bu_h ) \bI \: .
  \end{split}
\end{equation}
Inserting this into (\ref{eq:estimator_conforming_first}) leads to
\begin{equation}
  \begin{split}
    \| \bsigma_h^\Delta  \|_{\cA_\infty}^2
    & = \frac{1}{2 \mu} \| \bdev (\bsigma - \bsigma_h^R) \|^2 + 2 \mu \| \bepsilon (\bu - \bu_h) \|^2
    - \frac{2 \mu}{d} \| \div \: \bu_h \|^2 \\
    & - 2 ( \bsigma - \bsigma_h^R , \bepsilon (\bu - \bu_h) )
    - \frac{2}{d} ( \tr (\bsigma - \bsigma_h^R) , \div \: \bu_h ) \: .
  \end{split}
  \label{eq:estimator_conforming_second}
\end{equation}
The two last terms on the right-hand side of (\ref{eq:estimator_conforming_second}) can be treated as
\begin{equation}
  \begin{split}
    2 ( \bsigma - \bsigma_h^R , \bepsilon (\bu - \bu_h) )
    & = 2 ( \bsigma - \bsigma_h^R , \bnabla (\bu - \bu_h) ) - 2 ( \bsigma - \bsigma_h^R , \bas \: \bnabla (\bu - \bu_h) ) \\
    & = - 2 ( \div (\bsigma - \bsigma_h^R) , \bu - \bu_h ) + 2 ( \bas \: \bsigma_h^R , \bas \: \bnabla (\bu - \bu_h) ) \\
    & = 2 ( \bas \: \bsigma_h^R , \bnabla (\bu - \bu_h) ) \leq 2 C_K \| \bas \: \bsigma_h^R \| \: \| \bepsilon (\bu - \bu_h ) \| \\
    & \leq \frac{C_K^2}{\delta} \| \bas \: \bsigma_h^R \|^2 + \delta \| \bepsilon (\bu - \bu_h ) \|^2 \: ,
  \end{split}
  \label{eq:second_to_last_term_estimate}
\end{equation}
where the first estimate in Lemma \ref{lemma-bounds_asym_volum} is used (with $C_K$ depending only on the shape-regularity
of the triangulation) and $\delta > 0$ can be chosen arbitrarily.
The second estimate in Lemma \ref{lemma-bounds_asym_volum} leads to
\begin{equation}
  \begin{split}
    \frac{2}{d} ( \tr (\bsigma - \bsigma_h^R) , \div \: \bu_h )
    & \leq \frac{2}{d} C_A \| \bdev (\bsigma - \bsigma_h^R) \| \: \| \div \: \bu_h \| \\
    & \leq \frac{1}{2 \mu} \| \bdev (\bsigma - \bsigma_h^R) \|^2 + 2 \mu \left( \frac{C_A}{d} \right)^2 \| \div \: \bu_h \|^2 \: ,
  \end{split}
  \label{eq:last_term_estimate}
\end{equation}
where the constant $C_A$ again only depends on the shape-regularity of the triangulation. 

Combining (\ref{eq:estimator_conforming_second}) with (\ref{eq:second_to_last_term_estimate}) and (\ref{eq:last_term_estimate})
and using the fact that $\bas \: \bsigma_h^R = \bas \: \bsigma_h^\Delta$ leads to
\begin{equation}
  (2 \mu - \delta) \| \bepsilon (\bu - \bu_h) \|^2 \leq \| \bsigma_h^\Delta \|_{\cA_\infty}^2
  + 2 \mu \left( \frac{1}{d} + \left( \frac{C_A}{d} \right)^2 \right) \| \div \: \bu_h \|^2
  + \frac{C_K^2}{\delta} \| \bas \: \bsigma_h^\Delta \|^2 \: .
\end{equation}
Setting $\delta = \mu$, and noting that $2 \mu \| \bsigma_h^\Delta \|_{\cA_\infty}^2 = \| \bdev \: \bsigma_h^\Delta \|^2$ holds, we
finally obtain
\begin{equation}
  2 \mu \| \bepsilon (\bu - \bu_h) \|^2 \leq \frac{1}{\mu} \| \bdev \: \bsigma_h^\Delta \|^2
  + 4 \mu \left( \frac{1}{d} + \left( \frac{C_A}{d} \right)^2 \right) \| \div \: \bu_h \|^2
  + 2 \frac{C_K^2}{\mu} \| \bas \: \bsigma_h^\Delta \|^2 \: .
  \label{eq:guaranteed_upper_bound_incompressible}
\end{equation}
In the incompressible limit, our error estimator therefore consists element-wise of the three parts
\begin{equation}
  \eta_{A,T} = \frac{1}{(2 \mu)^{\frac{1}{2}}}\| \bdev \: \bsigma_h^\Delta \|_{T} \: , \:
  \eta_{B,T} = (2 \mu)^{\frac{1}{2}} \| \div \: \bu_h \|_{T} \: , \:
  \eta_{C,T} = \frac{1}{(2 \mu)^{\frac{1}{2}}} \| \bas \: \bsigma_h^\Delta \|_{T} \: .
\end{equation}
Together these provide a guaranteed upper bound for the energy norm of the error of the form
\begin{equation}
  2 \mu \| \bepsilon (\bu - \bu_h) \|^2 \leq 2 \sum_{T \in \cT_h} \eta_{A,T}^2
  + 2 \left( \frac{1}{d} + \left( \frac{C_A}{d} \right)^2 \right) \sum_{T \in \cT_h} \eta_{B,T}^2
  + 4 C_K^2 \sum_{T \in \cT_h} \eta_{C,T}^2
  \label{eq:guaranteed_upper_bound_incompressible_final}
\end{equation}
involving the controllable constants $C_A$ and $C_K$.

\section{Effect of the data approximation}

\label{sec-data_approximation}

In Section \ref{sec-error_estimation_general}, our a posteriori error estimator will be analyzed for the general case of arbitrary
Lam\'e parameter $\lambda$. The error will be estimated in the energy norm, expressed in terms of $\bu - \bu_h$ and $p - p_h$,
given by
\begin{equation}
  ||| ( \bu - \bu_h , p - p_h ) ||| = \left( 2 \mu \| \bepsilon (\bu - \bu_h) \|^2 + \frac{1}{\lambda} \| p - p_h \|^2 \right)^{1/2} \: .
  \label{eq:energy_norm}
\end{equation}
This section provides an investigation of the effect of the approximation of the right-hand side terms $\bff$ and $\bg$
on the solution $( \bu , p )$ of (\ref{eq:disp_pressure}). To this end, denote by $( \widetilde{\bu} , \tilde{p} )$ the solution of
(\ref{eq:disp_pressure}) with $\bff$ and $\bg$ replaced by $\cP_h^k \bff$ and $\cP_{h,\Gamma}^k \bg$, respectively. Then,
the difference $( \bu - \widetilde{\bu} , p - \tilde{p} )$ satisfies
\begin{equation}
  \begin{split}
    2 \mu ( \bepsilon (\bu - \widetilde{\bu}) , \bepsilon (\bv) ) + ( p - \tilde{p} , \div \: \bv ) 
    & = ( \bff - \cP_h^k \bff , \bv ) + \langle \bg - \cP_{h,\Gamma}^k \bg , \bv \rangle_{L^2 (\Gamma_N)} \: ,\\
    ( \div (\bu - \widetilde{\bu}) , q ) - \frac{1}{\lambda} ( p - \tilde{p} , q ) & = 0
  \end{split}
  \label{eq:disp_pressure_difference}
\end{equation}
for all $\bv \in H_{\Gamma_D}^1 (\Omega)^d$ and $q \in L^2 (\Omega)^d$. From the inf-sup stability, we deduce that
\begin{equation}
  ||| ( \bu - \tilde{\bu} , p - \tilde{p} ) ||| \lesssim
  \sup_{\bv \in H_{\Gamma_D}^1 (\Omega)^d} \frac{( \bff - \cP_h^k \bff , \bv )}{\| \bv \|_{H^1 (\Omega)}}
  + \sup_{\bv \in H_{\Gamma_D}^1 (\Omega)^d}
  \frac{\langle \bg - \cP_{h,\Gamma}^k \bg , \bv \rangle_{L^2 (\Gamma_N)}}{\| \bv \|_{H^1 (\Omega)}}
  \label{eq:difference_estimate}
\end{equation}
holds (cf. \cite[Theorem 4.2.3]{BofBreFor:13}), where $\lesssim$ denotes that the inequality holds up to a constant which
is independent of $\lambda$ (and, in the sequel, also of the local mesh-size $h_T$). Standard approximation estimates imply,
locally for each $T \in \cT_h$,
\begin{equation}
  \begin{split}
    ( \bff - \cP_h^k \bff , \bv )_T & = ( \bff - \cP_h^k \bff , \bv - \cP_h^k \bv )_T \\
    & \leq \| \bff - \cP_h^k \bff \|_T \| \bv - \cP_h^k \bv \|_T \lesssim h_T \| \bff - \cP_h^k \bff \|_T \| \bv \|_{H^1 (T)} \: .
  \end{split}
  \label{eq:approximation_f_local}
\end{equation}
Summing over all elements, this leads to
\begin{equation}
  \begin{split}
    ( \bff - \cP_h^k \bff , \bv ) & \lesssim \sum_{T \in \cT_h} h_T \| \bff - \cP_h^k \bff \|_T \| \bv \|_{H^1 (T)} \\
    & \leq \left( \sum_{T \in \cT_h} h_T^2 \| \bff - \cP_h^k \bff \|_T^2 \right)^{1/2} \| \bv \|_{H^1 (\Omega)} \: .
  \end{split}
  \label{eq:approximation_f_global}
\end{equation}
Similarly, for each $S \in \cS_h$ with $S \subseteq \Gamma_N$, we have
\begin{equation}
  \begin{split}
    \langle \bg - \cP_{h,\Gamma}^k \bg , \bv \rangle_S
    & = \langle \bg - \cP_{h,\Gamma}^k \bg , \bv - \cP_{h,\Gamma}^k \bv \rangle_S \\
    & \leq \| \bg - \cP_{h,\Gamma}^k \bg \|_S \| \bv - \cP_{h,\Gamma}^k \bv \|_S
    \lesssim h_S^{1/2} \| \bg - \cP_{h,\Gamma}^k \bg \|_S \| \bv \|_{H^{1/2} (S)} \: .
  \end{split}
  \label{eq:approximation_g_local}
\end{equation}
Summing over all sides in $\Gamma_N$, we obtain
\begin{equation}
  \begin{split}
    \langle \bg - \cP_{h,\Gamma}^k \bg , \bv \rangle_{\Gamma_N}
    & \lesssim \sum_{S \subseteq \Gamma_N} h_S^{1/2} \| \bg - \cP_{h,\Gamma}^k \bg \|_S \| \bv \|_{H^{1/2} (S)} \\
    & \leq \left( \sum_{S \subseteq \Gamma_N} h_S \| \bg - \cP_{h,\Gamma}^k \bg \|_S^2 \right)^{1/2} \| \bv \|_{H^{1/2} (\Gamma_N)}
    \\
    & \lesssim \left( \sum_{S \subseteq \Gamma_N} h_S \| \bg - \cP_{h,\Gamma}^k \bg \|_S^2 \right)^{1/2} \| \bv \|_{H^1 (\Omega)}
    \: ,
  \end{split}
  \label{eq:approximation_g_global}
\end{equation}
where the standard trace theorem from $H^1 (\Omega)^d$ to $H^{1/2} (\Gamma_N)^d$ is used. Finally,
inserting (\ref{eq:approximation_f_global}) and (\ref{eq:approximation_g_global}) into (\ref{eq:difference_estimate}) gives
\begin{equation}
  ||| ( \bu - \widetilde{\bu} , p - \tilde{p} ) ||| \lesssim
  \left( \sum_{T \in \cT_h} h_T^2 \| \bff - \cP_h^k \bff \|_T^2
  + \sum_{S \subseteq \Gamma_N} h_S \| \bg - \cP_{h,\Gamma}^k \bg \|_S^2 \right)^{1/2} \: .
  \label{eq:difference_estimate_final}
\end{equation}

We compare the convergence order of the local terms in the right-hand side in (\ref{eq:difference_estimate_final}) to the best
possible one for the local error $\| \bepsilon (\bu - \bu_h) \|_T$ of the approximation computed from
(\ref{eq:disp_pressure_discrete}). Assuming that $\bff \in H^\alpha (T)^d$ for some $\alpha \in ( 0 , k+1 )$, then we have
$\| \bff - \cP_h^k \bff \|_T \lesssim h_T^\alpha$, while the approximation error does, in general, behave like
$\| \bepsilon ( \bu - \bu_h ) \|_T = O (h_T^{1 + \alpha})$ at best. Note that $\bu$ can locally not be more than
$H^{2 + \alpha}$-regular, in general. Similarly, if we assume that $\bg \in H^\beta (S)^d$ for some $\beta \in ( 0 , k+1 )$, then
we have $\| \bg - \cP_{h,\Gamma}^k \bg \|_S \lesssim h_S^\beta$. The regularity of $\bu$, however, is locally not better than
$H^{3/2 + \beta}$, in general, leading to a convergence behavior not better than
$\| \bepsilon ( \bu - \bu_h ) \|_T = O (h_S^{1/2 + \beta})$ on elements adjacent to $S$. In any case, we get that
$||| ( \bu - \widetilde{\bu} , p - \tilde{p} ) ||| \lesssim ||| ( \bu - \bu_h , p - p_h ) |||$ independently of the triangulation. This is
completely similar to the situation for the Poisson equation treated in \cite[Theorem 4]{BraSch:08}.
We may therefore perform our analysis under the assumption that $\bff = \cP_h^k \bff$ and
$\bg = \cP_h^k \bg$ is fulfilled.

\section{A posteriori error estimation: The general case}

\label{sec-error_estimation_general}

We are now ready for the analysis of our error estimator in the general case. The definition of the stress directly leads to
\begin{equation}
  \begin{split}
    \tr \: \bsigma & = 2 \mu \div \: \bu + d p = \left( \frac{2 \mu}{\lambda} + d \right) p \: , \\
    \tr \: \bsigma_h & = 2 \mu \div \: \bu_h + d p_h = \left( \frac{2 \mu}{\lambda} + d \right) p_h
    + 2 \mu \left( \div \: \bu_h - \frac{1}{\lambda} p_h \right) \: ,
  \end{split}
\end{equation}
which implies
\begin{equation}
  \bepsilon (\bu) = \frac{1}{2 \mu} \left( \bsigma - p \bI \right)
  = \frac{1}{2 \mu} \left( \bsigma - \frac{\lambda}{2 \mu + d \lambda} (\tr \: \bsigma) \bI \right)
  =: \cA \bsigma
  \label{eq:strain_stress_relation}
\end{equation}
and
\begin{equation}
  \begin{split}
    \bepsilon (\bu_h) & = \frac{1}{2 \mu} \left( \bsigma_h - p_h \bI \right) \\
    & = \frac{1}{2 \mu} \left( \bsigma_h - \frac{\lambda}{2 \mu + d \lambda} (\tr \: \bsigma_h) \bI \right)
    + \frac{\lambda}{2 \mu + d \lambda} \left( \div \: \bu_h - \frac{1}{\lambda} p_h \right) \bI \\
    & = \cA \bsigma_h + \frac{\lambda}{2 \mu + d \lambda} \left( \div \: \bu_h - \frac{1}{\lambda} p_h \right) \bI \: .
  \end{split}
  \label{eq:strain_stress_relation_discrete}
\end{equation}
Note that (\ref{eq:strain_stress_relation}) and (\ref{eq:strain_stress_relation_discrete}) remain valid in the
incompressibe limit $\lambda \rightarrow \infty$, where $\cA$ tends to $\cA_\infty$ which was studied earlier in Section
\ref{sec-error_estimation_incompressible}.


Our a posteriori error estimator will be based on $\| \bsigma_h^\Delta \|_{\cA}^2$, the stress equilibration correction measured
with respect to the $\cA$-norm given by $\| \: \cdot \: \|_\cA := ( \: \cA (\cdot) \: , \: \cdot \: )^{1/2}$. Inserting the exact solution, we
obtain in analogy to (\ref{eq:estimator_conforming_first}) that
\begin{equation}
  \begin{split}
    \| \bsigma_h^\Delta \|_{\cA}^2 & = \| \bsigma_h^R - \bsigma_h (\bu_h , p_h)  \|_{\cA}^2
    = \| \bsigma - \bsigma_h^R - 2 \mu \bepsilon (\bu - \bu_h) - (p - p_h) \bI  \|_{\cA}^2 \\
    & = \| \bsigma - \bsigma_h^R \|_\cA^2 \\
    & + ( 2 \mu \bepsilon (\bu - \bu_h) + (p - p_h) \bI - 2 (\bsigma - \bsigma_h^R) ,
    \cA ( 2 \mu \bepsilon (\bu - \bu_h) + (p - p_h) \bI ) )
  \end{split}
  \label{eq:estimator_first}
\end{equation}
holds. The right term in the last inner product can be rewritten as
\begin{equation}
    \cA ( 2 \mu \bepsilon (\bu - \bu_h) + (p - p_h) \bI )
    = \bepsilon (\bu - \bu_h) + \frac{\lambda}{2 \mu + d \lambda} \left( \div \: \bu_h - \frac{p_h}{\lambda} \right) \bI \: .
\end{equation}
Inserting this into (\ref{eq:estimator_first}) leads to
\begin{equation}
  \begin{split}
    \| & \bsigma_h^\Delta \|_{\cA}^2 = \| \bsigma - \bsigma_h^R \|_\cA^2 + 2 \mu \| \bepsilon (\bu - \bu_h) \|^2
    + \frac{2 \mu \lambda}{2 \mu + d \lambda} \left( \frac{p}{\lambda} - \div \, \bu_h , \div \, \bu_h - \frac{p_h}{\lambda} \right) \\
    & + \left( p - p_h , \frac{p}{\lambda} - \div \: \bu_h \right)
    + \frac{d \lambda}{2 \mu + d \lambda} \left( p - p_h , \div \: \bu_h - \frac{p_h}{\lambda} \right) \\
    & - 2 ( \bsigma - \bsigma_h^R , \bepsilon (\bu - \bu_h) )
    - \frac{2 \lambda}{2 \mu + d \lambda} \left( \tr (\bsigma - \bsigma_h^R) , \div \: \bu_h - \frac{p_h}{\lambda} \right) \\
    & = \| \bsigma - \bsigma_h^R \|_\cA^2 + 2 \mu \| \bepsilon (\bu - \bu_h) \|^2 + \frac{1}{\lambda} \| p - p_h \|^2
    - \frac{2 \mu \lambda}{2 \mu + d \lambda} \left\| \div \: \bu_h - \frac{p_h}{\lambda} \right\|^2 \\
    & - 2 ( \bsigma - \bsigma_h^R , \bepsilon (\bu - \bu_h) )
    - \frac{2 \lambda}{2 \mu + d \lambda} \left( \tr (\bsigma - \bsigma_h^R) , \div \: \bu_h - \frac{p_h}{\lambda} \right) \\
    & = \| \bsigma - \bsigma_h^R \|_\cA^2 + ||| (\bu - \bu_h , p - p_h) |||^2
    - \frac{2 \mu \lambda}{2 \mu + d \lambda} \left\| \div \: \bu_h - \frac{p_h}{\lambda} \right\|^2 \\
    & - 2 ( \bsigma - \bsigma_h^R , \bepsilon (\bu - \bu_h) )
    - \frac{2 \lambda}{2 \mu + d \lambda} \left( \tr (\bsigma - \bsigma_h^R) , \div \: \bu_h - \frac{p_h}{\lambda} \right) \: ,
  \end{split}
  \label{eq:estimator_second}
\end{equation}
where we replaced $\div \: \bu$ by $p / \lambda$, wherever it occurred. From (\ref{eq:second_to_last_term_estimate}), we obtain
\begin{equation}
  2 ( \bsigma - \bsigma_h^R , \bepsilon (\bu - \bu_h) ) \leq \frac{C_K^2}{\delta} \| \bas \: \bsigma_h^R \|^2
  + \frac{\delta}{2 \mu} ||| ( \bu - \bu_h , p - p_h ) |||^2 \: ,
  \label{eq:second_to_last_term_estimate_general}
\end{equation}
which may be used to bound the second-to-last term in (\ref{eq:estimator_second}). For the last term in
(\ref{eq:estimator_second}), we deduce from (\ref{eq:bound_volum}) in Lemma \ref{lemma-bounds_asym_volum} and from
\begin{equation}
  \begin{split}
    \| \bsigma - \bsigma_h^R \|_{\cA}^2 & = ( \cA (\bsigma - \bsigma_h^R) , \bsigma - \bsigma_h^R ) \\
    & = \frac{1}{2 \mu}
    \left( \| \bsigma - \bsigma_h^R \|^2 - \frac{\lambda}{2 \mu + d \lambda} \| \tr (\bsigma - \bsigma_h^R) \|^2 \right) \\
    & \geq \frac{1}{2 \mu} \left( \| \bsigma - \bsigma_h^R \|^2 - \frac{1}{d} \| \tr (\bsigma - \bsigma_h^R) \|^2 \right)
    = \frac{1}{2 \mu} \| \bdev (\bsigma - \bsigma_h^R) \|^2
  \end{split}
  \label{eq:bound_A_dev}
\end{equation}
that
\begin{equation}
  \begin{split}
    \frac{2 \lambda}{2 \mu + d \lambda} & \left( \tr (\bsigma - \bsigma_h^R) , \div \: \bu_h - \frac{p_h}{\lambda} \right) \\
    & \leq \frac{2 \lambda C_A}{2 \mu + d \lambda} \| \bdev (\bsigma - \bsigma_h^R) \| \: \| \div \: \bu_h - \frac{1}{\lambda} p_h \| \\
    & \leq \frac{1}{2 \mu} \| \bdev (\bsigma - \bsigma_h^R) \|^2
    + \frac{2 \mu \lambda^2 C_A^2}{(2 \mu + d \lambda)^2} \| \div \: \bu_h - \frac{1}{\lambda} p_h \|^2 \\
    & \leq \| \bsigma - \bsigma_h^R \|_{\cA}^2
    + 2 \mu \left( \frac{\lambda C_A}{2 \mu + d \lambda} \right)^2 \| \div \: \bu_h - \frac{1}{\lambda} p_h \|^2
  \end{split}
  \label{eq:last_term_estimate_general}
\end{equation}
holds. Inserting (\ref{eq:second_to_last_term_estimate_general}) and (\ref{eq:last_term_estimate_general}) into
(\ref{eq:estimator_second}) and using the fact that $\bas \: \bsigma_h^R = \bas \: \bsigma_h^\Delta$ leads to
\begin{equation}
  \begin{split}
    \left( 1 - \frac{\delta}{2 \mu} \right) & ||| ( \bu - \bu_h , p - p_h ) |||^2 \\
    \leq \| \bsigma_h^\Delta \|_\cA^2
    & + \frac{2 \mu \lambda^2}{(2 \mu + d \lambda)^2} \left( \frac{2 \mu}{\lambda} + d + C_A^2 \right)
    \| \div \: \bu_h - \frac{p_h}{\lambda} \|^2 + \frac{C_K^2}{\delta} \| \bas \: \bsigma_h^\Delta \|^2 \: ,
  \end{split}
\end{equation}
where $\delta \in ( 0 , 1 )$ is still arbitrary. Setting again $\delta = \mu$, we finally obtain
\begin{equation}
  \begin{split}
    & ||| ( \bu - \bu_h , p - p_h ) |||^2 \\
    & \leq 2 \| \bsigma_h^\Delta \|_\cA^2
    + \frac{4 \mu \lambda^2}{(2 \mu + d \lambda)^2} \left( \frac{2 \mu}{\lambda} + d + C_A^2 \right)
    \| \div \: \bu_h - \frac{p_h}{\lambda} \|^2 + 2 \frac{C_K^2}{\mu} \| \bas \: \bsigma_h^\Delta \|^2 \: ,
  \end{split}
  \label{eq:guaranteed_upper_bound_general}
\end{equation}
Our error estimator therefore consists element-wise of the three parts
\begin{equation}
  \eta_{A,T} = \| \bsigma_h^\Delta \|_{\cA,T} \: , \: \eta_{B,T} = (2 \mu)^{1/2} \| \div \: \bu_h - \frac{p_h}{\lambda} \|_{T} \: , \:
  \eta_{C,T} = \frac{1}{(2 \mu)^{1/2}} \| \bas \: \bsigma_h^\Delta \|_{T} \: ,
  \label{eq:estimator_general}
\end{equation}
which together provide a guaranteed upper bound for the energy norm of the error.

We summarize the result of this derivation as follows.

\begin{theorem}
  Let $( \bu,p) \in H_{\Gamma_D}^1 (\Omega)^d \times L^2 (\Omega)$ be the exact solution of (\ref{eq:disp_pressure}) and
  $(\bu_h,p_h) \in \bV_h \times Q_h$ its finite element approximation satisfying (\ref{eq:disp_pressure_discrete}). Then,
  \begin{equation}
    \begin{split}
      ||| & ( \bu - \bu_h , p - p_h ) |||^2 \\
      & \leq 2 \sum_{T \in \cT_h} \eta_{A,T}^2
      + \frac{2 \lambda^2}{(2 \mu + d \lambda)^2} \left( \frac{2 \mu}{\lambda} + d + C_A^2 \right) \sum_{T \in \cT_h} \eta_{B,T}^2
      + 4 C_K^2 \sum_{T \in \cT_h} \eta_{C,T}^2 \: ,
    \end{split}
    \label{eq:guaranteed_upper_bound_general_final}
  \end{equation}
  involving the controllable constants $C_A$ and $C_K$ which only depend on the shape regularity of the triangulation.
\end{theorem}

Note that the term in front of the estimator contributions $\eta_{B,h}$ is monotonically increasing in $\lambda$ and therefore
bounded by its limit for $\lambda \rightarrow \infty$. Thus, (\ref{eq:guaranteed_upper_bound_general_final}) implies that
\begin{equation}
    ||| ( \bu - \bu_h , p - p_h ) |||^2 \leq 2 \sum_{T \in \cT_h} \eta_{A,T}^2
    + 2 \left( \frac{1}{d} + \frac{C_A^2}{d^2} \right) \sum_{T \in \cT_h} \eta_{B,T}^2
    + 4 C_K^2 \sum_{T \in \cT_h} \eta_{C,T}^2
  \label{eq:guaranteed_upper_bound_general_final_independent}
\end{equation}
holds which is independent of $\lambda$.

\section{Upper bound by a residual a posteriori error estimator and local efficiency}

\label{sec-local_efficiency}

Local efficiency of our equibrated error estimator (\ref{eq:estimator_general}) may be shown following the same idea as in
\cite{BraSch:08,BraPilSch:09} by bounding it from above with the residual estimator. To this end, we use the decomposition
(\ref{eq:patch_decomposition}) again and obtain
\begin{equation}
  \| \bsigma_h^\Delta \| \leq \sum_{z \in \cV_h^\ast} \| \bsigma_{h,z}^\Delta \|_{\omega_z^\ast} \: .
  \label{eq:local_estimate_upper_bound}
\end{equation}
The terms in the sum on the right-hand side in (\ref{eq:local_estimate_upper_bound}) can be treated by the following result.

\begin{proposition}
  Let $h_z$ denote the average diameter of all elements in $\omega_z^\ast$ and $h_S$ the diameter of the side $S$. Then,
  $\bsigma_{h,z}^\Delta \in \bSigma_{h,z}^\Delta$ minimizing $\| \bsigma_{h,z}^\Delta \|_{\omega_z^\ast}^2$ subject to
  (\ref{eq:equilibration_conditions_local}) satisfies
  \begin{equation}
    \| \bsigma_{h,z}^\Delta \|_{\omega_z^\ast} \lesssim h_z \| \bff + \div \: \bsigma_h ( \bu_h,p_h ) \|_{\omega_z^\ast}
    + \sum_{S \in \cS_{h,z}^\ast} h_S^{1/2} \| \llbracket \bsigma_h ( \bu_h,p_h ) \cdot \bn \rrbracket_S^\ast \|_S \: .
    \label{eq:efficiency_bound_patch}
  \end{equation}
  \label{prop-efficiency_bound_patch}
\end{proposition}

\begin{proof}
  {\em Step 1.} We first prove that
  \begin{equation}
    \| \bsigma_{h,z}^\Delta \|_{\omega_z^\ast}^2
    \lesssim h_z^{2-d} \left( | ( \div \: \bsigma_{h,z}^\Delta , \bz_{h,z} )_{\omega_z^\ast,h} |^2
    + \sum_{S \in \cS_{h,z}^\ast} | \langle \llbracket \bsigma_{h,z}^\Delta \cdot \bn \rrbracket_S , \bzeta_S \rangle_S |^2 \right)
    \label{eq:constraint_equivalence}
  \end{equation}
  holds for all $\bz_{h,z} \in \bZ_{h,z}$ with $\| \bz_{h,z} \|_{\omega_z^\ast,h}^2 \leq h_z^{2 d}$ and
  $\bzeta_S \in P_k (S)^d$ with $\| \bzeta_S \|_S^2 \leq h_S^{2 (d-1)}$, $S \in \cS_{h,z}^\ast$. To this end, we transform the vertex
  patch $\omega_z^\ast$ by a piecewise affine mapping onto a reference patch $\omega_{\rm ref}$ (e.g., centered at the origin
  and such that all edges attached to $z$ have unit length and all triangular angles at $z$ are equal). Due to the shape regularity
  of our triangulation $\cT_h$, this piecewise affine mapping possesses an inverse which we denote by $\bvarphi_z$.
  
  The space $\bSigma_{h,z}^\Delta$ has its counterpart $\bSigma_{\rm ref}^\Delta$ of functions $\bsigma_{\rm ref}^\Delta$
  defined on $\omega_{\rm ref}$ and connected via the Piola transformation
  \begin{equation}
    \bsigma_{h,z}^\Delta \circ \bvarphi_z = \frac{1}{\det (\bnabla \bvarphi_z)} \bsigma_{\rm ref}^\Delta (\bnabla \bvarphi_z)^T
    \label{eq:Piola_map}
  \end{equation}
  (cf. \cite[Sect. 2.1.3]{BofBreFor:13}). If we also define the test functions $\bz_{\rm ref} = \bz_{h,z} \circ \bvarphi_z$ and
  $\bzeta_{\rm ref} = \bzeta_S \circ \bvarphi_z$ on $\omega_{\rm ref}$, then
  \begin{equation}
    \begin{split}
      ( \div \: \bsigma_{h,z}^\Delta , \bz_{h,z} )_{\omega_z^\ast,h}
      & = ( \div \: \bsigma_{\rm ref}^\Delta , \bz_{\rm ref} )_{\omega_{\rm ref}} \\
      \langle \llbracket \bsigma_{h,z}^\Delta \cdot \bn \rrbracket_S , \bzeta_S \rangle_S
      & = \langle \llbracket \bsigma_{\rm ref}^\Delta \cdot \bn \rrbracket_{\hat{S}} , \bzeta_{\hat{S}} \rangle_{\hat{S}} \: , \:
      S = \bvarphi_z ( \hat{S} )
      \label{eq:Piola_inner}
    \end{split}
  \end{equation}
  holds (cf. \cite[Lemma 2.1.6]{BofBreFor:13}). On the reference patch $\omega_{\rm ref}$, we have
  \begin{equation}
    \| \bsigma_{\rm ref}^\Delta \|_{\omega_{\rm ref}}^2 \lesssim
    | ( \div \: \bsigma_{\rm ref}^\Delta , \bz_{\rm ref} )_{\omega_{\rm ref}} |^2
    + \sum_{\hat{S} \in \cS_{\rm ref}^\ast} | \langle \llbracket \bsigma_{\rm ref}^\Delta \cdot \bn \rrbracket_{\hat{S}} ,
    \bzeta_{\hat{S}} \rangle_{\hat{S}} |^2 \: ,
    \label{eq:equivalence_ref}
  \end{equation}
  since the right-hand side being zero forces the left-hand side to vanish and due to the finite dimension of the spaces involved
  and the fact that there is only a finite number of possible reference patches. The shape regularity implies that
  $| \bnabla \bvarphi_z | \lesssim h_z$ and $\det (\bnabla \bvarphi_z) \gtrsim h_z^d$ holds uniformly on $\omega_{\rm ref}$
  and therefore  
  \begin{equation}
    \| \bsigma_{h,z}^\Delta \|_{\omega_z^\ast,h}^2 \lesssim h_z^{2 - d} \| \bsigma_{\rm ref}^\Delta \|_{\omega_{\rm ref}}^2
    \label{eq:Piola_norm}
  \end{equation}
  follows directly from (\ref{eq:Piola_map}). Thus, (\ref{eq:constraint_equivalence}) follows from (\ref{eq:Piola_inner}),
  (\ref{eq:equivalence_ref}) and (\ref{eq:Piola_norm}).

  {\em Step 2.} Inserting the constraints (\ref{eq:equilibration_conditions_local}) into (\ref{eq:constraint_equivalence}) leads to
  \begin{equation}
    \begin{split}
      \| \bsigma_{h,z}^\Delta \|_{\omega_z^\ast,h}^2
      & \lesssim h_z^{2-d} \left( \left| ( (\bff + \div \: \bsigma_h (\bu_h,p_h)) \phi_z^\ast , \bz_{h,z} )_{\omega_z^\ast,h} \right|^2 \right. \\
      & \hspace{3.5cm} \left. + \sum_{S \in \cS_{h,z}^\ast}
      | \langle \llbracket \bsigma_h ( \bu_h,p_h ) \cdot \bn \rrbracket_S^\ast \phi_z^\ast , \bzeta_S \rangle_S |^2 \right) \: .
    \end{split}
  \end{equation}
  Combining the Cauchy-Schwarz inequality with our scaling of $\bz_{h,z}$ and $\bzeta_S$ implies
  \begin{equation}
    \begin{split}
      \| \bsigma_{h,z}^\Delta \|_{\omega_z^\ast,h}^2  
      & \lesssim h_z^2 \| (\bff + \div \: \bsigma_h (\bu_h,p_h)) \phi_z^\ast \|_{\omega_z^\ast,h}^2 \\
      & \hspace{3.2cm} + \sum_{S \in \cS_{h,z}^\ast} h_S \left( \frac{h_S}{h_z} \right)^{d-2} \| \llbracket \bsigma_h ( \bu_h,p_h ) \cdot \bn \rrbracket_S^\ast \phi_z^\ast \|_S^2 \\
      & \lesssim h_z^2 \| \bff + \div \: \bsigma_h (\bu_h,p_h) \|_{\omega_z^\ast,h}^2
      + \sum_{S \in \cS_{h,z}^\ast} h_S \| \llbracket \bsigma_h ( \bu_h,p_h ) \cdot \bn \rrbracket_S^\ast \|_S^2 \: ,
    \end{split}
    \label{eq:constraint_equivalence_inserted}
  \end{equation}
  where $h_S \lesssim h_z$ due to the shape regularity and the fact that $\phi_z^\ast$ is bounded by one is used.
  Taking the square root of (\ref{eq:constraint_equivalence_inserted}) implies (\ref{eq:efficiency_bound_patch}).
\end{proof}

The fact that
\begin{equation}
  \eta_{A,T}^2 + \eta_{C,T}^2 \lesssim \sum_{z \in T} \| \bsigma_{h,z}^\Delta \|_T^2
  \label{eq:bounds_A_and_C}
\end{equation}
is satisfied, combined with (\ref{eq:efficiency_bound_patch}), implies
\begin{equation}
  \begin{split}
    \eta_{A,T}^2 \: + \: & \eta_{B,T}^2 + \eta_{C,T}^2
    \lesssim \sum_{z \in T} \| \bsigma_{h,z}^\Delta \|_T^2 + \| \div \: \bu_h - \frac{p_h}{\lambda} \|_T^2 \\
    & \lesssim \sum_{T^\prime \subset \omega_T} h_{T^\prime}^2 \| \bff + \div \: \bsigma_h ( \bu_h,p_h ) \|_{T^\prime}^2
    + \sum_{S \in \cS_{h,z}^\prime} h_S \| \llbracket \bsigma_h ( \bu_h,p_h ) \cdot \bn \rrbracket_S^\ast \|_S^2 \\
    & \hspace{2cm} + \| \div \: \bu_h - \frac{p_h}{\lambda} \|_T^2
    \lesssim
    \sum_{T^\prime \subset \omega_T} \left( \eta_{R,T^\prime}^2 + h_{T^\prime}^2 \| \bff - \cP_h^k \bff \|_{T^\prime}^2 \right) \: ,
  \end{split}
  \label{eq:bound_residual}
\end{equation}
where $\omega_T = \cup \{ \omega_z : z \in T \}$ and where
\begin{equation}
  \begin{split}
    \eta_{R,T} & = \left( h_T^2 \| \cP_h^k \bff + \div \: \bsigma_h ( \bu_h,p_h ) \|_{T}^2
    + \sum_{S \subset \partial T} h_S \| \llbracket \bsigma_h ( \bu_h,p_h ) \cdot \bn \rrbracket_S^\ast \|_S^2 \right. \\
    & \left. \hspace{7cm} + \| \div \: \bu_h - \frac{p_h}{\lambda} \|_T^2 \right)^{1/2}
  \end{split}
  \label{eq:residual_estimator}
\end{equation}
denotes the residual error estimator. The local efficiency of this residual error estimator is shown, for the case of the
incompressible Stokes equations, in \cite[Sect. 4.10.3]{Ver:13}. In analogy to \cite[Theorem 4.70]{Ver:13} we obtain that
\begin{equation}
  \eta_{R,T} \lesssim \left( \| \bepsilon (\bu - \bu_h) \|_{\omega_T}^2 + \| p - p_h \|_{\omega_T}^2
  + h_T^2 \| \bff - \cP_h^k \bff \|_{\omega_T}^2 \right)^{1/2}
\end{equation}
holds. All together this leads to the local efficiency bound
\begin{equation}
  \eta_{A,T}^2 + \eta_{B,T}^2 + \eta_{C,T}^2 \lesssim
  \| \bepsilon (\bu - \bu_h) \|_{\tilde{\omega}_T}^2 + \| p - p_h \|_{\tilde{\omega}_T}^2
  + h_T^2 \| \bff - \cP_h^k \bff \|_{\tilde{\omega}_T}^2 \: ,
\end{equation}
where $\tilde{\omega}_T := \cup \{ \omega_{T^\prime} : T^\prime \subset \omega_T \}$, i.e., the next layer of elements around
$\omega_T$. 

\section{Numerical Results}

\label{sec-numerical}

Finally, we present numerical results obtained for a popular test example for linear elasticity computations. It is given by the Cook's
membrane problem which consists of a quadrilateral domain $\Omega \subset \R^2$ with corners $(0,0)$, $(0.48,0.44)$, $(0.48,0.6)$
and $(0,0.44)$. Homogeneous Dirichlet boundary conditions hold on the left boundary segment while traction forces are prescribed
on the remaining boundary parts, $\bg \equiv \bzero$ on the top and on the bottom, $\bg \equiv (0,0.01)$ on the right. We restrict
ourselves to the incompressible limit $\lambda = \infty$ since this is the most challenging situation.

\begin{figure}[!htb]
  \hspace{-0.75cm}\includegraphics[scale=0.75]{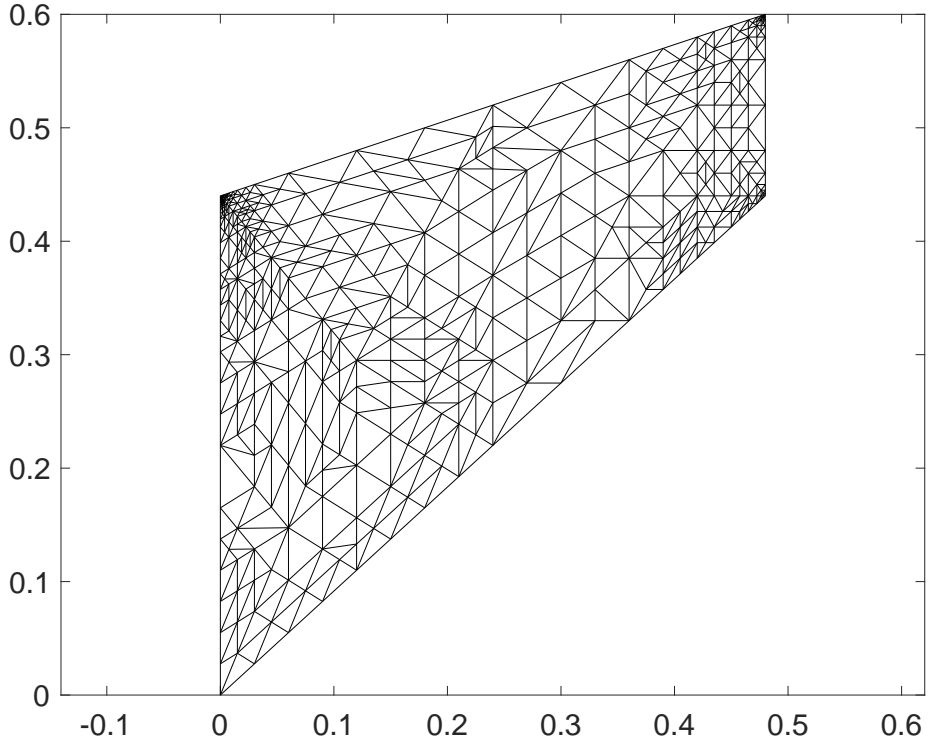}
  \caption{Triangulation after 9 adaptive refinement steps}
  \label{fig-cookmesh}
\end{figure}

Starting from an initial triangulation with 32 elements, 14 adaptive refinement steps are performed based on our
error estimator $\eta_T = \left( \eta_{A,T}^2 + \eta_{B,T}^2 + \eta_{C,T} \right)^{1/2}$ from (\ref{eq:estimator_general}).
The refinement strategy uses D\"orfler marking, i.e., a subset $\widetilde{\cT}_h \subset \cT_h$ of elements with the largest
estimator contributions is refined such that
\begin{equation}
  \left( \sum_{T \in \widetilde{\cT}_h} \eta_T^2 \right)^{1/2} \geq \theta \left( \sum_{T \in \cT_h} \eta_T^2 \right)^{1/2}
  \label{eq:bulk}
\end{equation}
holds. Figure \ref{fig-cookmesh} shows the refined triangulation after the 7th refinement step. As expected, most of the refinement
is concentrated around the most severe singularity at the left upper corner. However, some local refinement is also seen at the other
corners where the solution fails to be in $H^3 (\Omega)$. At later refinement steps this is no longer visible as nicely since individual
triangles can no longer be recognized in the vicinity of the corners. Figure \ref{fig-cookconv} shows the decrease of the error estimator
components $\eta_{A,T}$, $\eta_{B,T}$ and $\eta_{C,T}$ as well as the total estimator $\eta_T$ (on the vertical axis) in dependence
of the dimension of the finite element spaces (on the horizontal axis). All three estimator contributions apparently convergence with
the optimal rate $\eta \sim N^{-1}$, if $N$ denotes the associated number of degrees of freedom.

\begin{figure}[!htb]
  \hspace{-0.75cm}\includegraphics[scale=0.75]{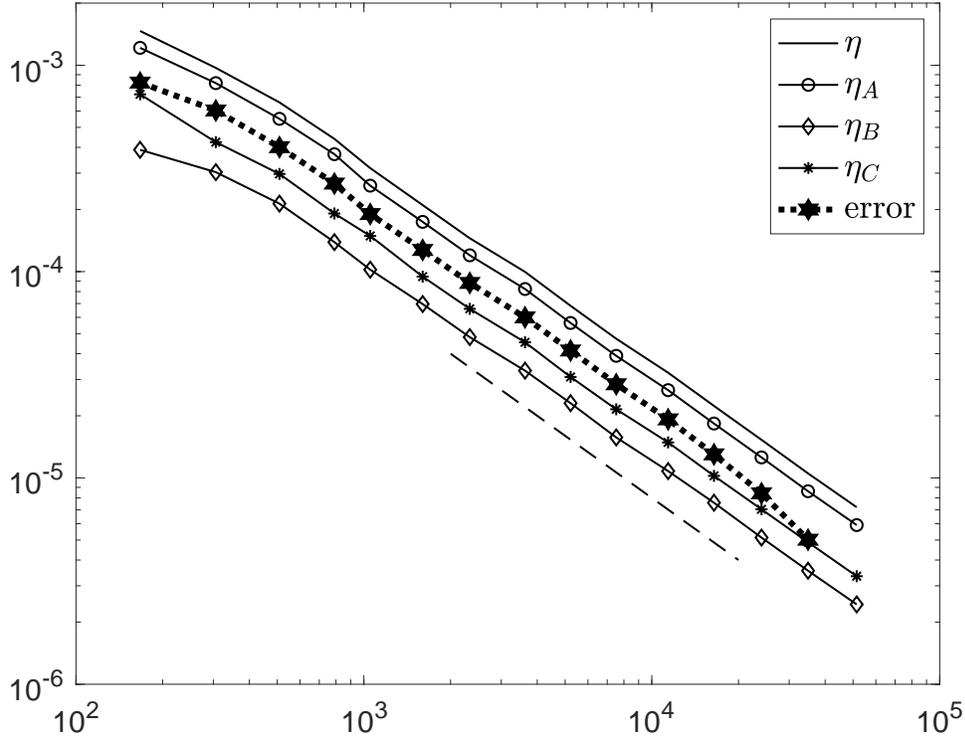}
  \caption{Error estimator convergence behavior}
  \label{fig-cookconv}
\end{figure}

In order to investigate the efficiency of our estimator, we also attempt a comparison with the actual true error $||| ( \bu - \bu_h , p - p_h ) |||$.
However, since the exact solution $( \bu , p )$ is not known to us analytically in this case, we use the approximation $||| ( \bu^\ast , p^\ast ) |||$
on the finest triangulation (after 14 refinements) instead and compute $||| ( \bu^\ast - \bu_h , p^\ast - p_h ) |||$. We may trust that
$||| ( \bu^\ast - \bu_h , p^\ast - p_h ) ||| \approx ||| ( \bu - \bu_h , p - p_h ) |||$ at least up to refinement level 12, before the curve starts to
astray downwards due to the discrepancy between $(\bu^\ast,p^\ast)$ and $(\bu,p)$. Figure \ref{fig-cookconv} also shows that the energy norm
of the error is already bounded from above by the dominating estimator contribution $\eta_A$ alone.

If one is interested in guaranteed upper bounds which are as tight as possible one may refine the derivation of the reliability of our error
estimator by incorporating the local constants $C_{K,z}$ and $C_{A,z}$ in (\ref{eq:Korn_average}) and (\ref{eq:dev_div_average}) into
the estimator contributions $\eta_{B,T}$ and $\eta_{C,T}$. To this end, the constants $C_{K,z}$ and $C_{A,z}$ may be bounded from
above as described at the end of Section \ref{sec-estimates_asym_volum}. However, since this becomes rather tedious we want to
finish our paper here with the conclusion that our numerical example already shows the potential of our equilibrated error estimator
for producing rather tight bounds for the error.

\newpage

\bibliography{../../biblio/articles,../../biblio/books}
\bibliographystyle{siamplain}

\end{document}